\documentclass[12pt]{amsart}
\usepackage{graphicx}
\usepackage{amscd}
\usepackage[top=1in,bottom=1in,left=1.15in,right=1.15in]{geometry}
\usepackage{amsmath,amssymb,amsfonts}
\usepackage{amsthm}
\usepackage{color}
\usepackage{hyperref}
\usepackage{tikz-cd}
\usetikzlibrary{plotmarks}
\usepackage{epstopdf}
\usetikzlibrary{matrix}
\usepackage{verbatim}
\usepackage{braket}
\usepackage{mathrsfs}
\newtheorem{theo}{Theorem}[section]
\newtheorem{prop}[theo]{Proposition}
\newtheorem{lemma}[theo]{Lemma}
\newtheorem{defn}[theo]{Definition}
\newtheorem{rem}[theo]{Remark}

\newtheorem{ex}[theo]{Example}

\allowdisplaybreaks

\usetikzlibrary{decorations.markings}

\tikzset{->-/.style={decoration={
  markings,
  mark=at position .5 with {\arrow{>}}},postaction={decorate}}}

\tikzset{-<-/.style={decoration={
  markings,
  mark=at position .5 with {\arrow{<}}},postaction={decorate}}}

\makeatletter
\@namedef{subjclassname@2020}{\textup{2020} Mathematics Subject Classification}
\makeatother

\begin{document}

\title[$\mathfrak{gl}(1 \vert 1)$-invariant, Reidemeister torsion and lens spaces]{$\mathfrak{gl}(1 \vert 1)$-Alexander polynomial for $3$-manifolds, Reidemeister torsion and lens spaces}
\author{Yuanyuan Bao}

\address{
Division of Mathematics \&
Research Center for Pure and Applied Mathematics,
Graduate School of Information Sciences,
Tohoku University, 6-3-09 Aramaki-Aza-Aoba, Aoba-ku, Sendai 980-8579, Japan
}
\email{yybao@tohoku.ac.jp}


\maketitle

\begin{abstract}
In this note, we reformulate the invariant $\Delta (M, \omega)$ that we defined before, and show its relation with Reidemeister torsion. We calculate $\Delta (M, \omega)$ when the $3$-manifolds are lens spaces, and discuss the classification of the lens spaces using $\Delta (M, \omega)$.
\end{abstract}

\section{Introduction}
A $3$-manifold in this paper indicates a connected  closed oriented smooth $3$-manifold. Kirby calculus \cite{MR467753} provides a way to study $3$-manifolds using framed links in the $3$-sphere $S^3$. 
Based on this theory, a combination of invariants for framed links defines a topological invariant for $3$-manifolds, if it is invariant under Kirby moves. 

Using Witten's ideas \cite{MR990772} from physics, Reshetikhin and Turaev \cite{MR1091619} provided the first rigorous construction of $3$-manifold invariants using linear sums of quantum invariants of framed links. Turaev \cite{MR3617439} further showed that from a modular category one can derive a Topological Quantum Field Theory, which includes $3$-manifold invariants. More recently, Turaev's theory has been further extended to many other types of categories, such as $G$-modular categories \cite{MR2674592}, relative $G$-premodular categories and so on. 

Costantino, Geer and Patureau \cite{MR3286896} proposed the concept relative $G$-premodular category and proved that the quantum invariant of framed links constructed from a relative $G$-premodular category can be used to define a $3$-manifold invariant. In particular, they constructed the $3$-manifold invariants $N_r$ based on a category $\mathscr{C}$, which is the tensor category of nilpotent representations of quantum $\mathfrak{sl}(2)$ at a root of unity of order $2r$ where $r\geq 2$ is not divisible by $4$.

Blanchet, Costantino, Geer and Patureau \cite{MR3539369} then constructed a new family of TQFTs based on the category $\mathscr{C}$. In particular when $r=2$, they showed that the invariant $Z_2$, which is a renormalization of $N_2$, is a canonical normalization of Reidemeister torsion. Thus their theory provides a TQFT for Reidemeister torsion. De Renzi \cite{MR4403930} further developed a general theory for the construction of Extended Topological Quantum Field Theories (ETQFTs) associated with the CGP invariants.

%Let $\mathscr{C}$ be a relative $G$-modular category. For a $3$-manifold $M$, a $\mathscr{C}$-ribbon graph $T$ and a cohomology class $\omega: H_1(M\backslash T, \mathbb{Z})\to G$ which satisfy some compatible conditions, let $L$ be a surgery presentation of $M$ with color induced from $\omega$. Then \cite{MR3286896} showed that the quantum invariant of $L\cup T$ after normalization is a topological invariant of $(M, T, \omega)$.

In \cite{MR4576428}, we followed the method in \cite{MR3286896} and constructed a CGP type invariant $\Delta (M, \omega)$ for a $3$-manifold $M$ and a non-trivial cohomology class $\omega$. The quantum invariant we used is Viro's $\mathfrak{gl}(1\vert 1)$-Alexander polynomial defined in \cite{MR2255851}. Briefly speaking, for a field $B$ of characteristic $0$ and a subgroup $G$ of the multiplicative group of $B$, there is a category $\mathcal{M}_B$ of finite dimensional modules over a $q$-less subalgebra $U^1$ of  quantum $\mathfrak{gl}(1 \vert 1)$. Viro's $\mathfrak{gl}(1\vert 1)$-Alexander polynomial is defined by using $\mathcal{M}_B$. Recently, Geer and Yong \cite{geer} defined an unrolled quantization $U_q^E (\mathfrak{gl}(1\vert 1))$ and discussed how to construct TQFTs and in particular $3$-manifold invariants from the category of weight modules of $U_q^E (\mathfrak{gl}(1\vert 1))$. It is not yet clear to us the relation between $\Delta (M, \omega)$ and the invariants defined in \cite{geer}.

In this note, we provide a reformulation of $\Delta(M, \omega)$. Using similar discussion as in Section 6.7 of \cite{MR3539369}, we show a  relation between $\Delta(M, \omega)$ and the Reidemeister torsion of $M$. In the last part, we calculate $\Delta(M, \omega)$  for lens spaces and discuss the classification of lens spaces. 

Although defined from different quantum groups, the invariants $N_2$ and $\Delta(M, \omega)$ share many similarities. We expect the existence of a TQFT for $\Delta(M, \omega)$ and it would be interesting to compare it with that of $N_2$. 

\medskip

\noindent{\bf Acknowledgements}
The author would like to thank Jun Murakami and Noboru Ito for helpful comments. The author was partially supported by JSPS KAKENHI Grant Number JP20K14304. 

%Consider $1$-palette $(B, G)$ where $G$ contains $\mathbb{Z}$ but no $\mathbb{Z}/2\mathbb{Z}$ as a subgroup.
%Let $M$ be a closed $3$-manifold. Consider a cohomology class $\omega: H_{1}(M, \mathbb{Z})\to G$. Let $L=L_1\cup L_2\cup \cdots \cup L_r$ be a surgery presentation of $M$. We say that $L$ is {\it computable} for $(M, \omega)$ if $\omega ([m])\neq 1\in G$ for any meridian $m$ of $L$. 

\section{Viro's $\mathfrak{gl}(1\vert 1)$-Alexander polynomial for a framed link}

In this section, we recall some basic facts and properties about Viro's $\mathfrak{gl}(1\vert 1)$-Alexander polynomial that will be used in Sections 3 and 4.

Viro \cite{MR2255851} defined a functor from the category of colored framed oriented trivalent graphs to the category of finite dimensional modules over a subalgebra $U^1$ of the $q$-deformed universal enveloping superalgebra $U_{q}(\mathfrak{gl}(1 \vert 1))$. Using this functor, in \cite[$\S$6]{MR2255851}, he defined the $\mathfrak{gl}(1\vert 1)$-Alexander polynomial for a trivalent graph. In this paper, most of the time, we only consider this polynomial for the case of links.%, so we review how it is defined for a link with a proper coloring. For the algebraic structures of $U^1$ and $U_{q}(\mathfrak{gl}(1 \vert 1))$, please read \cite[$\S$11: Appendix]{MR2255851}. 

The definition requires some data called $1$-palette. A {\it $1$-palette} (see \cite[$\S$2, 2.8]{MR2255851}) is a quadruple $$(B, G, W, G\times W \to G),$$ where $B$ is a commutative ring with unit, $G$ is a subgroup of the multiplicative group of $B$, $W$ is a subgroup of the additive group of $B$ which contains the unit $1$ of $B$, and $G\times W \to G: (t, N)\mapsto t^{N}$ is a bilinear map satisfying $t^1=t$ for each $t\in G$. 
%For each $t\in G$ satisfying $t^4\neq 1$ and $N\in W$, there exist two modules $U(t, N)_{+}$ and $U(t, N)_{-}$ of $U^1$. For the definition of $U(t, N)_{+}$ and $U(t, N)_{-}$, see \cite[$\S$11: Appendix]{MR2255851}.

In this paper, we assume that $B$ is a field of characteristic $0$. Let $G$ be a subgroup of the multiplicative group of $B$, which is an abelian group, and let $W=\mathbb{Z}$ and $G\times \mathbb{Z} \to G: (t, N)\mapsto t^{N}$. It is easy to see that $(B, G, \mathbb{Z}, G\times \mathbb{Z} \to G)$ becomes a $1$-palette, and we simply use $(B, G)$ to denote it because the meanings of $W$ and the bilinear map are clear. 

Let $L=L_1\cup L_2 \cup \cdots \cup L_r$ be an $r$-component oriented framed link in $S^3$. Given a $1$-palette $(B, G)$, a {\it color} of $L$ is a map 
\begin{eqnarray*}
c: \{L_i\}_{i=1}^r &\to& G\times \mathbb{Z}\\
L_i &\mapsto& (t_i, N_i).
\end{eqnarray*}
The first coordinate $t_i$ is called the {\it multiplicity} and the second one $N_i$ is called the {\it weight}. For the pair $(L, c)$, Viro defined the $\mathfrak{gl}(1\vert 1)$-Alexander polynomial $\Delta_c(L) \in B$. For the precise definition, please see \cite{MR2255851}, and note that its notation in \cite{MR2255851} is $\underline{\Delta}^1(L)$. 

%\begin{rem}
%A component $L_i$
%\end{rem}

\subsection{Properties}
Here we review several properties of $\Delta_c(L)$ to be used in Sections 3 and 4, most of the proofs of which can be found in \cite{MR2255851}. 

\begin{lemma}
\label{inverse}
Let $L=L_1\cup L_2\cup \cdots \cup L_r$ be an oriented framed link with color $c(L_i)=(t_i, 0)$ for $i=1, 2, \cdots, r$. Let $L'$ be the link obtained from $L$ by revising the orientation of a certain component $L_i$, and let $c'$ be the color of $L'$ which is the same as $c$ except that it defines color $(t_i^{-1}, 0)$ on $L_i$. Then we have
$$\Delta_c(L)=-\Delta_{c'}(L').$$
\end{lemma}

The lemma was not clearly stated in \cite{MR2255851}, but its proof is based on the fact that simultaneously reversing the orientation and reversing the multiplicity of a link component does not change the Boltzmann weights, as stated in the end of \cite[7.7.F.bis]{MR2255851}. We remind that the Boltzmann weights in the third and forth columns of \cite[Table 3]{MR2255851} around a crossing contains some typos.

As stated in Section 7.5 of \cite{MR2255851}, introducing nonzero weights and framings does not enrich the Alexander polynomial of a link. The influence of changing weights was given in 7.5A. The following lemma is a special case of it.

\begin{lemma}
\label{weight}
Suppose $c$ is a color of an oriented framed link $L=L_1\cup L_2\cup \cdots \cup L_r$ so that $c(L_i)=(t_i, 1)$ for $i=1, 2, \cdots, r$. Let $c_0$ be the color of $L$ so that $c_0(L_i)=(t_i, 0)$. Then we have 
$$\Delta_c(L)=\prod_{i=1}^{r}t_i^{-2(\sum_{j=1}^r lk_{ij})}\Delta_{c_0}(L),$$
where $lk_{ij}$ denotes the linking number between $L_i$ and $L_j$ for $i\neq j$, and $lk_{ii}$ is the framing of $L_i$.
\end{lemma}

The behavior under connected sum can be obtained easily from the definition.

\begin{lemma}
\label{sum}
Suppose $(L, c)$ and $(L', c')$ are two colored framed links. Assume that there are two components $L_i\subset L$ and $L'_j\subset L'$ so that $c(L_i)=c'(L'_j)=(t, N)$.
Let $L\sharp L'$ be the connected sum of $L$ and $L'$ along $L_i$ and $L'_j$, and $c\sharp c'$ be the color of $L\sharp L'$ inherited from $c$ and $c'$. Then we have
$$\Delta_{c\sharp c'}(L\sharp L')=(t^2-t^{-2})\Delta_c(L)\Delta_{c'}(L').$$
\end{lemma}

\medskip

Now we consider an example which we will use in Section 4.

\begin{ex}
\rm
\label{lens}
Let $L=L_1\cup L_2\cup \cdots \cup L_n$ be the link defined in Fig. \ref{fig1}, where the $i$-th component is colored by $(t_i, 0)$. The $a_i$ inside a box indicates that the framing of the component is $a_i$. This link is a conneced sum of $n-1$ Hopf links. Then by Lemma~\ref{sum}, we have 
$$\Delta_c(L)=\prod_{i=2}^{n-1}(t_i^2-t_i^{-2})\Delta_c(L_{1}\cup L_2)\Delta_c(L_{2}\cup L_3)\cdots \Delta_c(L_{n-1}\cup L_n).$$
%where $L_{i-1}\cup L_i$ is the Hopf link obtained from $L$ by ignoring the other components of $L$. 
Note that since the weights are zero, the framings do not affact the Alexander polynomial, and as calculated in \cite[Example 2.2]{MR4576428} we have $\Delta_c(L_{i-1}\cup L_i)=1$ for $i=2,\cdots, n$. Therefore 
$$\Delta_c(L)=\prod_{i=2}^{n-1}(t_i^2-t_i^{-2}).$$
\end{ex}

\begin{figure}
\begin{center}
\includegraphics[width=80mm]{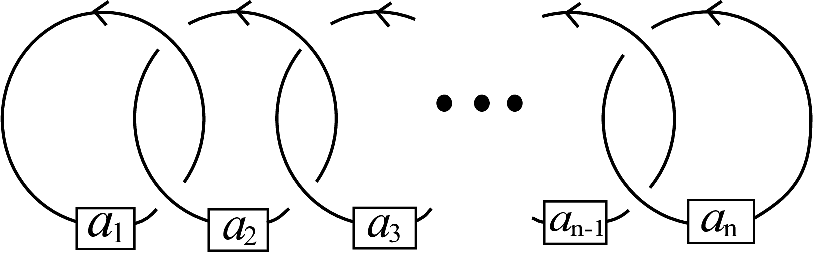}
\caption{A connected sum of Hopf links}
\label{fig1}
\end{center}
\end{figure}

The next lemma is the relation of $\Delta_c(L)$ and the Conway function of $L$.

\begin{lemma}[7.7G in \cite{MR2255851}]
\label{conway}
For a positive integer $r$, let $G_r$ be the free abelian group generated by $t_1, t_2, \cdots, t_r$ and let $B_r$ be the field of rational functions of $t_1, t_2, \cdots, t_r$. For a framed link $L=L_1\cup L_2\cup \cdots \cup L_r$, consider the color $c$ of $L$ defined by $c(L_i)=(t_i, 0)$. Then we have
$$\Delta_c(L)=\nabla(L)(t_1^{2}, t_2^{2}, \cdots, t_r^{2}),$$
where $\nabla(L)$ is the Conway function of $L$.
\end{lemma}

\subsection{Kirby color}
Now we explain how to extend the definition to the case that each link component has a Kirby color, which was defined in \cite[Definition 3.3]{MR4576428}. 
For a $1$-palette $(B, G)$ and $t\in G$, a {\it Kirby color} $\Omega(t, 1)$ is defined to be a formal linear sum $$\Omega(t, 1)=d(t) (t, 1) + d(t^{-1}) (t^{-1}, 1), $$ where $\displaystyle d(t)=\frac{1}{t^2-t^{-2}}$. Suppose a link $L=L_1\cup L_2\cup \cdots \cup L_r$ has a kirby color $kc$ so that $kc(L_i)=\Omega(t_i, 1)$ for $1\leq i \leq r$. We define $\Delta_{kc}(L)$ by replacing each component with a formal linear sum illustrated as follows.
\begin{align*}
\left(\begin{tikzpicture}[baseline=-0.65ex, thick, scale=1]
\draw (0.5,-0.5) [<-]  to (0.5,1.2);
\draw (0.5,-0.8) node {$\Omega(t_i, 1)$};
\end{tikzpicture}\right):=d(t_i)\left(\begin{tikzpicture}[baseline=-0.65ex, thick, scale=1]
\draw (0.5,-0.5) [<-]  to (0.5,1.2);
\draw (0.5,-0.8) node {$(t_i, 1)$};
\end{tikzpicture}\right)+d(t_i^{-1})\left(\begin{tikzpicture}[baseline=-0.65ex, thick, scale=1]
\draw (0.5,-0.5) [->]  to (0.5,1.2);
\draw (0.5,-0.8) node {$(t_i^{-1}, 1)$};
\end{tikzpicture}\right),
\end{align*}
where a colored diagram inside round brackets represents the $\mathfrak{gl}(1\vert 1)$-Alexander polynomial of it. When $L$ has $r$ components, $\Delta_{kc}(L)$ is a linear sum of $2^r$ terms.

\begin{ex} 
\rm
In the following diagram of a Hopf link, each component has a Kirby color. By definition we have
\begin{eqnarray*}
\left (\begin{tikzpicture}[baseline=-0.65ex, thick, scale=1]
\draw (-0.45,1) node {$\scriptstyle m$};
\draw (-0.3,1.1) to (-0.6,1.1);
\draw (-0.3,0.89) to (-0.3,1.11);
\draw (-0.3,0.90) to (-0.6,0.90);
\draw (-0.6,0.89) to (-0.6,1.11);
\draw (-0.3,1) [<-]to (0.3,0.4);
\draw (-0.3,0.4) to (-0.1,0.6);
\draw (0.1,0.8) [->] to (0.3,1);
\draw (-0.3,0.4) to (0.25,-0.09);
\draw (0.45,1) node {$\scriptstyle n$};
\draw (0.6,1.1) to (0.3,1.1);
\draw (0.6,0.89) to (0.6,1.11);
\draw (0.6,0.90) to (0.3,0.90);
\draw (0.3,0.89) to (0.3,1.11);
\draw (0.1,0.2) to (0.3,0.4);
\draw  (-0.6,1) arc (90:315:0.6);
\draw  (0.6,1) arc (90:-140:0.6);
\draw (0.5,-0.5) node {$\scriptstyle \Omega(v, 1)$};
\draw (-0.5,-0.5) node {$\scriptstyle \Omega(u, 1)$};
\end{tikzpicture}\right )
= d(u)d(v)\left[
\left (\begin{tikzpicture}[baseline=-0.01ex, thick, scale=1]
\draw (-0.45,1) node {$\scriptstyle m$};
\draw (-0.3,1.1) to (-0.6,1.1);
\draw (-0.3,0.89) to (-0.3,1.11);
\draw (-0.3,0.90) to (-0.6,0.90);
\draw (-0.6,0.89) to (-0.6,1.11);
\draw (-0.3,1) [<-]to (0.3,0.4);
\draw (-0.3,0.4) to (-0.1,0.6);
\draw (0.1,0.8) [->] to (0.3,1);
\draw (-0.3,0.4) to (0.25,-0.09);
\draw (0.45,1) node {$\scriptstyle n$};
\draw (0.6,1.1) to (0.3,1.1);
\draw (0.6,0.89) to (0.6,1.11);
\draw (0.6,0.90) to (0.3,0.90);
\draw (0.3,0.89) to (0.3,1.11);
\draw (0.1,0.2) to (0.3,0.4);
\draw  (-0.6,1) arc (90:315:0.6);
\draw  (0.6,1) arc (90:-140:0.6);
\draw (0.5,-0.5) node {$\scriptstyle (v, 1)$};
\draw (-0.5,-0.5) node {$\scriptstyle (u, 1)$};
\end{tikzpicture}\right)
-
\left (\begin{tikzpicture}[baseline=-0.01ex, thick, scale=1]
\draw (-0.45,1) node {$\scriptstyle m$};
\draw (-0.3,1.1) to (-0.6,1.1);
\draw (-0.3,0.89) to (-0.3,1.11);
\draw (-0.3,0.90) to (-0.6,0.90);
\draw (-0.6,0.89) to (-0.6,1.11);
\draw (-0.3,1) [->]to (0.3,0.4);
\draw (-0.3,0.4) to (-0.1,0.6);
\draw (0.1,0.8) [->] to (0.3,1);
\draw (-0.3,0.4) to (0.25,-0.09);
\draw (0.45,1) node {$\scriptstyle n$};
\draw (0.6,1.1) to (0.3,1.1);
\draw (0.6,0.89) to (0.6,1.11);
\draw (0.6,0.90) to (0.3,0.90);
\draw (0.3,0.89) to (0.3,1.11);
\draw (0.1,0.2) to (0.3,0.4);
\draw  (-0.6,1) arc (90:315:0.6);
\draw  (0.6,1) arc (90:-140:0.6);
\draw (0.5,-0.5) node {$\scriptstyle (v, 1)$};
\draw (-0.5,-0.5) node {$\scriptstyle (u^{-1}, 1)$};
\end{tikzpicture}\right)\right.\\
\left.-
\left (\begin{tikzpicture}[baseline=-0.01ex, thick, scale=1]
\draw (-0.45,1) node {$\scriptstyle m$};
\draw (-0.3,1.1) to (-0.6,1.1);
\draw (-0.3,0.89) to (-0.3,1.11);
\draw (-0.3,0.90) to (-0.6,0.90);
\draw (-0.6,0.89) to (-0.6,1.11);
\draw (-0.3,1) [<-]to (0.3,0.4);
\draw (-0.3,0.4) to (-0.1,0.6);
\draw (0.1,0.8) [<-] to (0.3,1);
\draw (-0.3,0.4) to (0.25,-0.09);
\draw (0.45,1) node {$\scriptstyle n$};
\draw (0.6,1.1) to (0.3,1.1);
\draw (0.6,0.89) to (0.6,1.11);
\draw (0.6,0.90) to (0.3,0.90);
\draw (0.3,0.89) to (0.3,1.11);
\draw (0.1,0.2) to (0.3,0.4);
\draw  (-0.6,1) arc (90:315:0.6);
\draw  (0.6,1) arc (90:-140:0.6);
\draw (0.5,-0.5) node {$\scriptstyle (v^{-1}, 1)$};
\draw (-0.5,-0.5) node {$\scriptstyle (u, 1)$};
\end{tikzpicture}\right)
+
\left(\begin{tikzpicture}[baseline=-0.01ex, thick, scale=1]
\draw (-0.45,1) node {$\scriptstyle m$};
\draw (-0.3,1.1) to (-0.6,1.1);
\draw (-0.3,0.89) to (-0.3,1.11);
\draw (-0.3,0.90) to (-0.6,0.90);
\draw (-0.6,0.89) to (-0.6,1.11);
\draw (-0.3,1) [->]to (0.3,0.4);
\draw (-0.3,0.4) to (-0.1,0.6);
\draw (0.1,0.8) [<-] to (0.3,1);
\draw (-0.3,0.4) to (0.25,-0.09);
\draw (0.45,1) node {$\scriptstyle n$};
\draw (0.6,1.1) to (0.3,1.1);
\draw (0.6,0.89) to (0.6,1.11);
\draw (0.6,0.90) to (0.3,0.90);
\draw (0.3,0.89) to (0.3,1.11);
\draw (0.1,0.2) to (0.3,0.4);
\draw  (-0.6,1) arc (90:315:0.6);
\draw  (0.6,1) arc (90:-140:0.6);
\draw (0.5,-0.5) node {$\scriptstyle (v^{-1}, 1)$};
\draw (-0.5,-0.5) node {$\scriptstyle (u^{-1}, 1)$};
\end{tikzpicture}\right)\right].
\end{eqnarray*}
\end{ex}

\section{$\mathfrak{gl}(1\vert 1)$-Alexander polynomial for a $3$-manifold.}
%A $3$-manifold in this paper indicates a connected compact closed oriented smooth $3$-manifold. 

Consider a $1$-palette $(B, G)$ where $G$ is a finitely generated abelian group containing $\mathbb{Z}$ but no $\mathbb{Z}/2\mathbb{Z}$ as a subgroup. Namely, $G$ has free summand but does not have any elements of order $2$.
%Essentially it is enough to consider a group without elements of order $2$, and if it does not have free summand, we simply take a direct sum of it with $\mathbb{Z}$ and extend $B$ accordingly. 

For a 3-manifold $M$ and a non-trivial cohomology class $\omega: H_{1}(M, \mathbb{Z})\to G$, we constructed an invariant $\Delta(M; \omega)$ in \cite{MR4576428}. Note that the original definition also considered a ribbon graph $\Gamma$ in $M$, but here we take $\Gamma$ to be the empty set.

The definition is as follows. Let $L=L_1\cup L_2\cup \cdots \cup L_r$ be a surgery presentation of $M$. We say that $L$ is {\it computable} for $(M, \omega)$ if for any $1\leq i \leq r$, $\omega ([m_i])\neq 1\in G$, where $[m_i]$ is the homology class of an oriented meridian of $L_i$. We showed in \cite{MR4576428} that the value in Definition \ref{invariant} does not change under Kirby moves, and thus becomes a $3$-manifold invariant.

\begin{defn}
\label{invariant}
\rm 
For a $3$-manifold $M$ and a non-trivial cohomology class $\omega: H_{1}(M, \mathbb{Z})\to G$, let $L=L_1\cup L_2\cup \cdots \cup L_r$ be a computable surgery presentation of $(M, \omega)$. Then 
$$\Delta (M, \omega):=\frac{\Delta_{kc} (L)}{2^{r}(-1)^{\sigma_{+}(L)}} \in B$$ is called the {\it $\mathfrak{gl}(1\vert 1)$-Alexander polynomial} for $(M, \omega)$, where $\sigma_{+}(L)$ is the number of positive eigenvalues of the linking matrix of $L$ and $kc$ is the kirby color so that  $kc(L_i)=\Omega(t_i, 1)$ with $t_i=\omega([m_i])$ for $1\leq i \leq r$.

\end{defn}

Here we show that the invariant above can be reformulated as follows.

\begin{prop}
\label{refine}
Let $L=L_1\cup L_2\cup \cdots \cup L_r$ be a computable surgery presentation of $(M, \omega)$ as above, and let $t_i=\omega([m_i])$ for $i=1, 2,\cdots, r$. Then
$$\Delta (M, \omega)=(-1)^{\sigma_{+}(L)}\prod_{i=1}^r d(t_i) \,\Delta_{c} (L), $$
where $\displaystyle d(t_i)=\frac{1}{t_i^2-t_i^{-2}}$ and $c$ is the color of $L$ defined by $c(L_i)=(t_i, 0)$.
\end{prop}
\begin{proof}
For each component $L_i$ of $L$, let $-L_i$ be the same knot with reversed orientation. For $\boldsymbol {\epsilon}=(\epsilon_1, \epsilon_2, \cdots, \epsilon_r)\in \{1, -1\}^r$, let $\boldsymbol {\epsilon}L=\epsilon_1L_1\cup \epsilon_2L_2 \cup \cdots \cup \epsilon_rL_r$, where $\epsilon_i L_i=L_i$ (resp. $\epsilon_i L_i=-L_i$) if $\epsilon_i=1$ (resp. $\epsilon_i=-1$), and let $\boldsymbol {\epsilon}c$ be the color of $\boldsymbol {\epsilon}L$ so that $\boldsymbol {\epsilon}c(\epsilon_i L_i)=(t_i^{\epsilon_i}, 1)$. We further define $\boldsymbol {\epsilon}c_0$ to be the color of $\boldsymbol {\epsilon}L$ so that $\boldsymbol {\epsilon}c_0(\epsilon_i L_i)=(t_i^{\epsilon_i}, 0)$. Then by Lemma~\ref{weight}, we have
\begin{eqnarray*}
\Delta_{\boldsymbol {\epsilon}c}(\boldsymbol {\epsilon}L)&=&\prod_{i=1}^{r}t_i^{-2\epsilon_i(\sum_{j=1}^r lk(\epsilon_i L_i, \epsilon_j L_j))}\Delta_{\boldsymbol {\epsilon}c_0}(\boldsymbol {\epsilon}L),
\end{eqnarray*}
where $lk(\epsilon_i L_i, \epsilon_j L_j))$ denotes the linking number of $\epsilon_i L_i$ and $\epsilon_j L_j$ for $i\neq j$ and the framing of $\epsilon_i L_i$ when $i=j$. 
Note that 
\begin{eqnarray*}
\prod_{i=1}^{r}t_i^{-2\epsilon_i(\sum_{j=1}^r lk(\epsilon_i L_i, \epsilon_j L_j))}=\prod_{i=1}^{r}t_i^{-2\epsilon_i(\sum_{j=1}^r \epsilon_i\epsilon_j  lk( L_i,  L_j))}
= (\prod_{i=1}^{r}\prod_{j=1}^{r}t_i^{ \epsilon_j  lk( L_i,  L_j)})^{-2}\\
=\left ( \prod_{j=1}^{r}(\prod_{i=1}^{r}t_i^{ lk( L_i,  L_j)})^{\epsilon_j}\right )^{-2}=\left (\prod_{j=1}^{r}\omega^{\epsilon_j}(\prod_{i=1}^{r}[m_i]^{ lk( L_i,  L_j)})\right )^{-2}=\left (\prod_{j=1}^{r}\omega^{\epsilon_j}(1)\right )^{-2}=1. 
\end{eqnarray*}
Here we use the fact that $\prod_{i=1}^{r}[m_i]^{ lk( L_i,  L_j)}=1\in H_1(M, \mathbb{Z})$ for any $1\leq j \leq r$. As a result, we have $\Delta_{\boldsymbol {\epsilon}c}(\boldsymbol {\epsilon}L)=\Delta_{\boldsymbol {\epsilon}c_0}(\boldsymbol {\epsilon}L)$.

In the definition of $\Delta_{kc} (L)$, each component $L_i$ of $L$ has the Kirby color $\Omega(t_i, 1)$. Recall the definition in Section 2.2, we have
\begin{eqnarray*}
\Delta_{kc} (L)&=&\sum_{\boldsymbol {\epsilon}=(\epsilon_1, \epsilon_2, \cdots, \epsilon_r)\in \{1, -1\}^r} d(t_1^{\epsilon_1})d(t_2^{\epsilon_2})\cdots d(t_r^{\epsilon_r})\Delta_{\boldsymbol {\epsilon}c}(\boldsymbol {\epsilon}L)\\
&=&\sum_{\boldsymbol {\epsilon}\in \{1, -1\}^r} \epsilon_1\epsilon_2\cdots \epsilon_rd(t_1)d(t_2)\cdots d(t_r)\Delta_{\boldsymbol {\epsilon}c_0}(\boldsymbol {\epsilon}L)\\
&=&\sum_{\boldsymbol {\epsilon}\in \{1, -1\}^r} \epsilon_1\epsilon_2\cdots \epsilon_rd(t_1)d(t_2)\cdots d(t_r)\epsilon_1\epsilon_2\cdots \epsilon_r \Delta_{c_0}(L)\\
&=& 2^r d(t_1)d(t_2)\cdots d(t_r)\Delta_{c_0}(L),
\end{eqnarray*}
where $c_0$ is the color of $L$ so that $c_0(L_i)=(t_i, 0)$ for $i=1, 2,\cdots, r$. The third equality follows from Lemma~\ref{inverse}.

Therefore $$\Delta (M, \omega):=\frac{\Delta_{kc} (L)}{2^{r}(-1)^{\sigma_{+}(L)}}=\frac{d(t_1)d(t_2)\cdots d(t_r)\Delta_{c_0}(L)}{(-1)^{\sigma_{+}(L)}}=(-1)^{\sigma_{+}(L)}\prod_{i=1}^r d(t_i)\Delta_{c_0}(L).$$ Here $c_0$ is exactly the $c$ in the proposition.

\end{proof}

\section{Relation with Reidemeister torsion}
\subsection{Reidemeister torsion}
For a $3$-manifold $M$ and a ring homomorphism $\phi$ from $\mathbb{Z}[H_{1}(M, \mathbb{Z})]$ to a field $B$ such that $\phi(H_{1}(M, \mathbb{Z})) \neq 1$, one can define the Reidemeister-Franz torsion $\tau^{\phi}(M) \in B$, a precise definition of which can be found in \cite{MR1958479}. It is only well-defined up to multiplication of $\pm\phi(h)$ with $h\in H_{1}(M, \mathbb{Z})$. The ambiguity can be fixed by introducing a smooth Euler structure $e$ and a homology orientation $\mathcal{O}$. Namely $\tau^{\phi}(M, e, \mathcal{O})$ is a well-defined value for $(M, \phi, e, \mathcal{O})$. Since $M$ is an oriented 3-manifold, we can choose the natural homology orientation of $M$. 

When $M$ is obtained by surgery on a framed oriented link $L=L_1\cup L_2\cup \cdots \cup L_r$, Turaev found a surgery formula for the  Reidemeister-Franz torsion of $M$, which can be found in VIII.2 of \cite{MR1958479}. 

\begin{theo}[Turaev]
\label{turaev}
Let $L=L_1\cup L_2\cup \cdots \cup L_r$ be a surgery presentation of a $3$-manifold $M$.
For a ring homomorphism $\phi$ from $\mathbb{Z}[H_{1}(M, \mathbb{Z})]$ to a field $B$ and a Euler structure $e$, we have 
$$\tau^{\phi}(M, e, \mathcal{O})=(-1)^{\sigma_{+}(L)}\prod_{i=1}^r \frac{1}{(t_i-1)}\cdot t_1^{k_1/2}t_2^{k_2/2}\cdots t_r^{k_r/2} \nabla(L)(t_1^{1/2}, t_2^{1/2}, \cdots, t_r^{1/2}).$$
Here $t_i=\phi([m_i])$, $k=(k_1, k_2, \cdots, k_r)$ is the charge of $L$ corresponding to the Euler structure $e$, and $\mathcal{O}$ is the natural homology orientation of $M$.
\end{theo}

\begin{rem}
For a framed link $L=L_1\cup L_2\cup \cdots \cup L_r$, as stated in VII.2 of \cite{MR1958479}, the product $t_1^{k_1/2}t_2^{k_2/2}\cdots t_r^{k_r/2} \nabla(L)(t_1^{1/2}, t_2^{1/2}, \cdots, t_r^{1/2})$ is a rational function of $t_1, t_2, \cdots, t_r$. After replacing $t_i$ with $\phi([m_i])$ in both numerator and denominator, we get an element of $B$. The formula in Theorem \ref{turaev} should be understood in this way, where we abuse the notation $t_i$. 
\end{rem}

\subsection{Relation}
The following proposition states a relation between $\Delta(M, \omega)$ and the Reidemeister torsion. For a $3$-manifold $M$ and a $1$-palette $(B, G)$, a non-trivial cohomology class $\omega: H_{1}(M, \mathbb{Z})\to G$ can be extended linearly to a ring homomorphism $\phi_{\omega}: \mathbb{Z}[H_{1}(M, \mathbb{Z})] \to B$. Suppose $p$ is the homomorphism from $G$ to itself which sends $g\in G$ to $g^4$. The homomorphism $p \circ \omega: H_{1}(M, \mathbb{Z})\to G$ can be extended to a ring homomorphism $\phi_{p\circ \omega}: \mathbb{Z}[H_{1}(M, \mathbb{Z})] \to B$.

\begin{prop}
We have
\begin{eqnarray*}
\tau^{\phi_{p\circ \omega}}(M, e, \mathcal{O})=[\omega(c(e))]^2 \Delta (M, \omega),
\end{eqnarray*}
where $c(e)\in H_1(M, \mathbb{Z})$ is the homology class so that $e=c(e)e^{-1}$.
\end{prop}
\begin{proof}
Let $L=L_1\cup L_2\cup \cdots \cup L_r$ be a computable surgery presentation of $M$, and let $t_i=\omega([m_i])$ for $1\leq i \leq r$.
By Theorem \ref{turaev}, Lemma \ref{conway} and Proposition \ref{refine}, we have
\begin{eqnarray*}
\tau^{p\circ \omega}(M, e, \mathcal{O})&=& (-1)^{\sigma_{+}(L)}\prod_{i=1}^r \frac{1}{(t_i^4-1)}\cdot t_1^{2k_1}t_2^{2k_2}\cdots t_r^{2k_r} \nabla(L)(t_1^{2}, t_2^{2}, \cdots, t_r^{2})\\
&=& (-1)^{\sigma_{+}(L)}\prod_{i=1}^r \frac{t_i^{2(k_i-1)}}{(t_i^{2}-t_i^{-2})}\Delta_c(L)=\prod_{i=1}^r  t_i^{2(k_i-1)} \Delta (M, \omega).
\end{eqnarray*}
The relation $\displaystyle \prod_{i=1}^r  t_i^{(k_i-1)}=\omega(c(e))$ is discussed in VI 2.3 of \cite{MR1958479}. 
\end{proof}

\begin{rem}
\cite[Theorem 6.23]{MR3539369} showed a relation of the Reidemeister torsion and the invariant $Z_2$, and thus provided a normalization of Reidemeister torsion. The invariant $\Delta (M, \omega)$ can also be regarded as a normalization, but we have a deficit that only the Reidemeister torsions corresponding to special type of ring homormorphisms can be considered.
\end{rem}

\section{Lens spaces}
Using a similar discussion as in \cite[Prop. 6.24]{MR3539369}, we calculate $\Delta (M, \omega)$ when $M$ is a lens space. Let $(B, G)$ be a $1$-palette for which $G$ has free summand but has no element of order $2$. 

\begin{prop}
\label{lensvalue}
Let $p>q>0$ be two coprime integers and consider the lens space $L(p, q)$. Consider a $1$-palette $(B, G)$ for which there is a non-trivial cohomology class $\omega: \mathbb{Z}/ p\mathbb{Z} \to G$. Using the surgery presentation in Example \ref{lens} to compute $\Delta (L(p, q), \omega)$, we have 
\begin{eqnarray*}
\Delta (L(p, q), \omega)&=&(-1)^n\frac{1}{t_1^2-t_1^{-2}}\frac{1}{t_n^2-t_n^{-2}}\\&=&-\frac{1}{t_n^2-t_n^{-2}} \frac{1}{t_n^{2q}-t_n^{-2q}},
\end{eqnarray*}
where $t_i=\omega([m_i])$ for $i=1, n$.
\end{prop}
\begin{proof}
Consider a continued fraction $p/q=[a_{1}, a_{2}, \cdots, a_{n}]$ where each $a_i\geq 2$. Namely,
$$p/q=a_{1}-\cfrac{1}{a_{2}-\cfrac{1}{a_{3}-\cdots-\cfrac{1}{a_{n}}}}.$$
Then the surgery presentation in Fig. \ref{fig1} provides the lens space $L(p, q)$, and the linking matrix of the surgery presentation is 
$$\begin{pmatrix}
a_1 & 1 &&  &\\
1 & a_2 & 1 &\text{\huge{0}}  & \\
 & \ddots & \ddots & \ddots& \\
 &   & \ddots & \ddots& 1\\
 & \text{\huge{0}} &  & 1 & a_n\\
\end{pmatrix},$$ which is positive definite.

By Proposition~\ref{refine} and Example~\ref{lens}, we have
\begin{eqnarray*}
\Delta (L(p, q), \omega)=(-1)^{n}\prod_{i=1}^n d(t_i)\Delta_c(L)&=&(-1)^{n}\prod_{i=1}^n d(t_i)\prod_{i=2}^{n-1}(t_i^2-t_i^{-2})\\
&=&(-1)^n\frac{1}{t_1^2-t_1^{-2}}\frac{1}{t_n^2-t_n^{-2}},
\end{eqnarray*} 
where the color $c$ is defined by $c(L_i)=(t_i, 0)=(\omega([m_i]), 0)$ for $i=1, 2, \cdots, n$.

From the linking matrix, we see that $t_i$'s satisfy the following relations $t_{i-1}t_i^{a_i}t_{i+1}=1$ for $1\leq i \leq n$ if we define $t_0=t_{n+1}=1$. From the relations we see that $t_1=t_n^{(-1)^{n-1}q}$. 
Indeed, define $c_{n+1}=0, c_n=1$ and recursively $c_i=-a_{i+1}c_{i+1}-c_{i+2}$ for $0\leq i \leq n-1$. Form the relations $t_{i-1}t_i^{a_i}t_{i+1}=1$ we see that $t_i=t_n^{c_i}$ for $1\leq i \leq n$. The calculation of $c_1$ has been given by \cite[Prop. 6.24]{MR3539369}, which is $c_1=(-1)^{n-1}q$.

Therefore we have
$$\Delta (L(p, q), \omega)=(-1)^n\frac{1}{t_n^2-t_n^{-2}}\frac{1}{t_n^{(-1)^{n-1}2q}-t_n^{-(-1)^{n-1}2q}}=-\frac{1}{t_n^2-t_n^{-2}} \frac{1}{t_n^{2q}-t_n^{-2q}}.$$

\end{proof}

The classification of lens spaces was first given by Reidemeister \cite{MR3069647} and Franz \cite{MR1581473} by using Reidemeister torsions. Brody \cite{MR116336} and more recently Przytychi and Yasuhara \cite{MR1988423} also provided different proofs. In \cite{MR3539369}, the non-semi-simple quantum invariant $Z_2$ is used to classify lens spaces. In the remaining part, we want to explore the classification problem using $\Delta (M, \omega)$.

\begin{theo}[Classification of lens spaces]
Let $(p,q)$ and $(p, q)$ be two pairs of coprime integers with $p>0$. The lens spaces $L(p, q)$ and $L(p, q')$ are orientation-preservingly diffeomorphic if and only if either $q \equiv q'  \mod p$ or $qq' \equiv 1 \mod p$.
\end{theo}

Franz's lemma, which we recall below, plays a key role in the proof of the classification of lens spaces.

\begin{lemma}[Franz \cite{MR1581473}]
\label{franz}
Let $p\geq 3$ be an integer and let $\mathbb{Z}_p$ be the multiplicative group of invertible elements of $\mathbb{Z}/ p\mathbb{Z}$. Suppose a sequence of integers $(a_i)_{i\in \mathbb{Z}_p}$ satisfies the following conditions
\begin{enumerate}
\item $\displaystyle \sum_{i\in \mathbb{Z}_p} a_i=0$,
\item $a_i=a_{-i}$, for all $i\in  \mathbb{Z}_p$,
\item $\displaystyle \prod_{i\in \mathbb{Z}_p} (\xi^i -1)^{a_i}=1$ for any $p$-th root of unity $\xi$.
\end{enumerate}
Then $a_i=0$ for all $i\in \mathbb{Z}_p$.
\end{lemma}

Then we can just follow the proofs in \cite{MR3069647, MR3539369} to show the following proposition.

\begin{prop}
The invariants $\Delta (M, \omega)$ distinguish the lens spaces $L(p, q)$ when $p$ is odd.
\end{prop}
\begin{proof}
Let $L(p, q)$ and $L(p, q')$ be two lens spaces with $p>0$. When $q=q' \mod p$ or $qq'\equiv 1 \mod p$, a precise orientation preserving diffeomorphism exists between $L(p, q)$ and $L(p, q')$. Now we can assume that $p>q>0$ and $p>q'>0$, and show that if $L(p, q)\cong L(p, q')$ we have either $q=q'$ or $qq'\equiv 1 \mod p$. 

Let $B=\mathbb{C}$, and let $G$ be the direct sum of $\mathbb{Z}$ and the group generated by $\eta=\exp (\frac{2\pi i}{p})$. There are totally $p-1$ non-trivial homomorphisms from $\mathbb{Z}/ p\mathbb{Z}$ to $G$. For the lens space $L(p, q)$, consider the surgery presentation in Example \ref{lens} and let $[m]$ be the meridian of the right-most component. Let $\omega_k$ be the map sending $[m]$ to $\eta^k$ for $1\leq k \leq p-1$. We consider the corresponding invariants $\Delta (M, \omega_k)$ and by Proposition \ref{lensvalue} we have
$$\Delta (L(p, q), \omega_k)=-\frac{1}{\eta^{2k}-\eta^{-2k}} \frac{1}{\eta^{2kq}-\eta^{-2kq}}.$$
Similarly, for the lens space $L(p, q')$, we define $w_l': \mathbb{Z}/ p\mathbb{Z} \to G$ to be the homomorphism sending the right-most meridian $[m']$ to  
$\eta^l$ for $1\leq l \leq p-1$, and we have
$$\Delta (L(p, q'), \omega'_l)=-\frac{1}{\eta^{2l}-\eta^{-2l}} \frac{1}{\eta^{2lq'}-\eta^{-2lq'}}.$$

When $L(p, q)\cong L(p, q')$, there is an invertible element $c\in \mathbb{Z}/p\mathbb{Z}$ such that
$$\Delta (L(p, q), \omega_k)=\Delta (L(p, q'), \omega'_{kc}).$$
Namely
$$-\frac{1}{\eta^{2k}-\eta^{-2k}} \frac{1}{\eta^{2kq}-\eta^{-2kq}}=-\frac{1}{\eta^{2kc}-\eta^{-2kc}} \frac{1}{\eta^{2kcq'}-\eta^{-2kcq'}},$$
for $1\leq k \leq p-1$. Then we have 
$$(\eta^{2k}-\eta^{-2k})(\eta^{2kq}-\eta^{-2kq})=(\eta^{2kc}-\eta^{-2kc})(\eta^{2kcq'}-\eta^{-2kcq'}).$$
Taking the product with its conjugacy on both sides we have
$$(1-\eta^{-4k})(1-\eta^{4k})(1-\eta^{-4kq})(1-\eta^{4kq})=(1-\eta^{-4kc})(1-\eta^{4kc})(1-\eta^{-4kcq'})(1-\eta^{4kcq'}).$$
Note that since $p$ is an odd number, $\{\eta^{4k}\}_{1 \leq k \leq p-1}$ contains all the $p$-th root of unity. Namely we have 
$$(1-\xi^{-1})(1-\xi)(1-\xi^{-q})(1-\xi^{q})=(1-\xi^{-c})(1-\xi^{c})(1-\xi^{-cq'})(1-\xi^{cq'}),$$
for all $p$-th root of unity $\xi$. 

By Lemma~\ref{franz} we see that $\{1, -1, q, -q\}$ and $\{c, -c, cq', -cq\}$ must coincide as sets. If $c=\pm 1$ we have $q=q'$. If $c=\pm q$ we have $qq'\equiv \pm 1 \mod p$. But from
$(\eta^2-\eta^{-2})(\eta^{2q}-\eta^{-2q})=(\eta^{2c}-\eta^{-2c})(\eta^{2cq'}-\eta^{-2cq'})$
we see that $\eta^2-\eta^{-2}=\eta^{2qq'}-\eta^{-2qq'}$. Since $\eta=\exp (\frac{2\pi i}{p}) $ and $p$ is odd, we have $qq'\equiv 1 \mod p$.
\end{proof}

\begin{rem}
When $p$ is even, the invariants $\Delta (M, \omega)$ in general can not distinguish $L(p, q)$. For example, $L(12, 5)$ and $L(12, 11)$ are not orientation preservingly homeomorphic, but there are no suitable $\Delta (M, \omega)$ to distinguish them. The reason is that the only candidate of $(B, G)$ must have the group $G$ isomorphic to $\mathbb{Z}/3\mathbb{Z}$, and for any $t\in \mathbb{Z}/3\mathbb{Z}$ we have $t^{10}-t^{-10}=t^{22}-t^{-22}$. Furthermore, when $p$ is a power of $2$, the invariant $\Delta (M, \omega)$ can not be defined. In this case, it is impossible to find a group $G$ so that there is a non-trivial cohomology class $\omega$.
\end{rem}

\bibliographystyle{siam}
\bibliography{bao}

\end{document}